\documentclass[reqno]{amsart}
\usepackage{hyperref}
\usepackage{url}
\usepackage{amssymb}

\newtheorem{theorem}{Theorem}[section]

\newtheorem{conj}{Conjecture}[section]
\newtheorem{lemma}{Lemma}[section]

\theoremstyle{definition}

\theoremstyle{remark}

\numberwithin{equation}{section}

\begin{document}

\title[Explicit Formulas for Partition Pairs and Triples with 3-Cores]
 {Explicit Formulas for Partition Pairs and \\ Triples with 3-Cores}

\author{LIUQUAN WANG}
\address{Department of Mathematics, National University of Singapore, Singapore, 119076, SINGAPORE}

\email{wangliuquan@u.nus.edu; mathlqwang@163.com}



\date{July 2, 2015}
\dedicatory{}

\maketitle

\begin{abstract}
Let $A_{3}(n)$ (resp.\ ${{B}_{3}}(n)$) denote the number of partition pairs (resp.\ triples) of $n$ where each partition is 3-core. By applying Ramanujan's ${}_{1}\psi_{1}$ formula and Bailey's ${}_{6}\psi_{6}$ formula, we find the explicit formulas for $A_{3}(n)$ and $B_{3}(n)$. Using these formulas, we confirm a conjecture of Xia and establish many arithmetic identities satisfied by $A_{3}(n)$ and $B_{3}(n)$.
\end{abstract}

\section{Introduction}	
A partition of a positive integer $n$ is any nonincreasing sequence of positive integers whose sum is $n$. For example, $6=3+2+1$ and $\lambda =\{3,2,1\}$ is a partition of 6. A partition $\lambda $ of $n$ is said to be a $t$-core if it has no hook numbers that are multiples of $t$. We denote the number of $t$-core partitions of $n$ by ${{a}_{t}}(n)$.

The generating function of ${{a}_{t}}(n)$ is given by (see \cite[Eq.\ (2.1)]{Garvan1990})
\begin{equation}\label{atgen}
\sum\limits_{n=0}^{\infty }{{{a}_{t}}(n){{q}^{n}}}=\frac{(q^t;q^t)_{\infty}^{t}}{(q;q)_{\infty}},
\end{equation}
here and throughout this paper, we use the following notation
\[{{(a;q)}_{\infty }}:=\prod\limits_{n=0}^{\infty }{(1-a{{q}^{n}})}, \quad (a;q)_{n}:=\frac{(a;q)_{\infty}}{(aq^n;q)_{\infty}} \quad(-\infty <n <\infty).\]
For convenience, we also introduce the brief notation
\[(a_1,a_2, \cdots, a_n;q)_{\infty}:=(a_1;q)_{\infty}(a_2;q)_{\infty}\cdots (a_n;q)_{\infty}.\]

A partition $k$-tuple $({{\lambda }_{1}},{{\lambda }_{2}},\cdots ,{{\lambda }_{k}})$ of $n$ is a $k$-tuple of partitions ${{\lambda }_{1}},{{\lambda }_{2}},\cdots ,{{\lambda }_{k}}$ such that the sum of all the parts equals $n$. For example, let ${{\lambda }_{1}}=\{2,1\},{{\lambda }_{2}}=\{1,1\},{{\lambda }_{3}}=\{1\}$. Then $(\lambda_1, \lambda_2)$ is a partition pair of 5 since $2+1+1+1=5$, and $({{\lambda }_{1}},{{\lambda }_{2}},{{\lambda }_{3}})$ is a partition triple of $6$ since $2+1+1+1+1=6$. A partition $k$-tuple of $n$ with $t$-cores is a partition $k$-tuple  $({{\lambda }_{1}},{{\lambda }_{2}},\cdots ,{{\lambda }_{k}})$ of $n$ where each $\lambda _{i}$ is $t$-core for $i=1,2, \cdots ,k$.

Let ${{A}_{t}}(n)$ (resp.\ ${{B}_{t}}(n)$) denote the number of partition pairs (resp.\ triples) of $n$ with $t$-cores. From (\ref{atgen}) we know the generating functions for ${{A}_{t}}(n)$ and ${{B}_{t}}(n)$ are
\begin{equation}\label{At}
\sum\limits_{n=0}^{\infty }{{{A}_{t}}(n){{q}^{n}}}=\frac{(q^t;q^t)_{\infty}^{2t}}{(q;q)_{\infty}^{2}}
\end{equation}
and
\begin{equation}\label{Bt}
\sum\limits_{n=0}^{\infty }{{{B}_{t}}(n){{q}^{n}}}=\frac{(q^t;q^t)_{\infty}^{3t}}{(q;q)_{\infty}^{3}}
\end{equation}
respectively.

In this paper, we focus on partition $k$-tuples with 3-cores for $1\le k \le 3$. The properties of $a_{3}(n)$, $A_{3}(n)$ and $B_{3}(n)$ have drawn much attention in the past years. In 1996, using the tools of modular forms, Granville and Ono \cite{Ono} first discovered the following formula for $a_{3}(n)$:
\begin{equation}\label{a3n}
{{a}_{3}}(n)={{d}_{1,3}}(3n+1)-{{d}_{2,3}}(3n+1),
\end{equation}
where ${{d}_{r,3}}(n)$ denote the number of positive divisors of $n$ congruent to $r$ modulo 3.

In 2009, by using some known identities, Hirschhorn and Sellers \cite{Hirschhorn} provided an elementary proof of (\ref{a3n}). Moreover, let
\[3n+1= \prod\limits_{p_i\equiv 1 \, \textrm{(mod 3)}}{p_{i}^{\alpha_i}} \cdot \prod\limits_{q_j\equiv 2 \, \textrm{(mod 3)}}{q_{j}^{\beta_j}}\]
with each $\alpha_i, \beta_j \ge 0$ be the prime factorization of $3n+1$, they gave the explicit formula:
\begin{displaymath}
a_{3}(n)=\left\{\begin{array}{ll}
\prod(\alpha_i+1) & \textrm{if all $\beta_j$ are even;}\\
0 &\textrm{otherwise.}
\end{array} \right.
\end{displaymath}
Some arithmetic identities were then obtained as corollaries. For example, let $p\equiv 2$ (mod 3) be a prime and let $k$ be a positive even integer. Then, for all $n\ge 0$,
\[{{a}_{3}}\Big({{p}^{k}}n+\frac{{{p}^{k}}-1}{3}\Big)={{a}_{3}}(n).\]

In 2014, Lin \cite{Lin} found some arithmetic relations about ${{A}_{3}}(n)$ such as ${{A}_{3}}(8n+6)=7{{A}_{3}}(2n+1)$. By using some theta function identities, Baruah and Nath \cite{BaruahJNT} established three infinite families of arithmetic identities involving ${{A}_{3}}(n)$. For any integer $k\ge 1$, they proved that
\begin{equation}\label{BNid}
\begin{split}
A_{3}\bigg(2^{2k+2}n+\frac{2(2^{2k}-1)}{3}\bigg)&=\frac{2^{2k+2}-1}{3}A_{3}(4n), \\
{{A}_{3}}\bigg({{2}^{2k+2}}n+\frac{2({{2}^{2k+2}}-1)}{3}\bigg)&=\frac{{{2}^{2k+2}}-1}{3}\cdot{{A}_{3}}(4n+2)-\frac{{{2}^{2k+2}}-4}{3}\cdot{{A}_{3}}(n), \\
A_{3}\bigg(2^{2k+1}n+\frac{5\cdot 2^{2k}-2}{3}\bigg)&=\big(2^{2k+1}-1\big)A_{3}(2n+1).
\end{split}
\end{equation}

Xia \cite{XiaRama} found several infinite families of congruences modulo 4, 8 for $A_{3}(n)$. For example, he showed that for all integers $n \ge 0$,
\[A_{3}(8n+4)\equiv 0 \pmod{4}, \quad A_{3}(16n+4)\equiv 0 \pmod{8}.\]
He also proposed the following conjecture
\begin{conj}\label{Xiaconj}
For any positive integer $j$ and prime $p$, there exists a positive integer $k_{0}$ such that for all $n\ge 0$ and $\alpha \ge 0$,
\[A_{3}\Big(4^{k_{0}(\alpha+1)}n+\frac{2^{2k_{0}(\alpha+1)-1}-2}{3}\Big) \equiv 0 \pmod{p^{j}}.\]
\end{conj}

For more results about ${{a}_{3}}(n)$ and ${{A}_{3}}(n)$, see \cite{Baruah2}--\cite{BaruahJNT} and \cite{Hirschhorn,Lin,XiaRama,Yao}.
	
Recently, the author \cite{WangIJNT} studied the arithmetic properties of $B_{3}(n)$. By elementary $q$ series manipulations, we found three infinite families of arithmetic identities satisfied by ${{B}_{3}}(n)$. For any integer $k\ge 1$, we proved that
\begin{equation}\label{B3id1}
B_{3}(3^{k}n+3^{k}-1)=3^{2k}B_{3}(n),
\end{equation}
\begin{equation}\label{B3id2}
{{B}_{3}}\big({{2}^{k+1}}n+{{2}^{k}}-1\big)=\frac{{{2}^{2k+2}}+{{(-1)}^{k}}}{5}\cdot{{B}_{3}}(2n),
\end{equation}
and
\begin{equation}\label{B3id3}
{{B}_{3}}\big({{2}^{k+1}}n+{{2}^{k+1}}-1\big)=\frac{{{2}^{2k+2}}+{{(-1)}^{k}}}{5}\cdot{{B}_{3}}(2n+1)+\frac{{{2}^{2k+2}}-4{{(-1)}^{k}}}{5}\cdot{{B}_{3}}(n).
\end{equation}

In viewing of (\ref{a3n}), it is natural to ask that whether we can find the explicit formulas for $A_{3}(n)$ and $B_{3}(n)$ or not. In this paper, we give a positive answer to this question. By using Ramanujan's ${}_{1}\psi_{1}$ summation formula and Bailey's ${}_{6}\psi_{6}$ formula, we give a new simple proof of (\ref{a3n}) and find the explicit formulas for $A_{3}(n)$ and $B_{3}(n)$. With these formulas in mind, most of the results mentioned above become direct consequences. In particular, we will confirm Conjecture \ref{Xiaconj} and give some generalizations of (\ref{BNid})--(\ref{B3id3}).

\section{Explicit Formula for $A_{3}(n)$}
Before we present the explicit formula for $A_{3}(n)$, we provide a new elementary proof of (\ref{a3n}). The key tool in this section is Ramanujan's ${}_{1}\psi_{1}$ summation formula \cite[Theorem 1.3.12]{Berndt}.
\begin{lemma}[Ramanujan's ${}_{1}\psi_{1}$ Summation]\label{Rama}
 For $|b/a|<|z|<1$ and $|q|<1$,
\begin{equation}\label{psi}
\sum\limits_{n=-\infty}^{\infty}{\frac{(a;q)_{n}}{(b;q)_{n}}z^{n}=\frac{(az,q/(az),q,b/a;q)_{\infty}}{(z,b/(az),b,q/a;q)_{\infty}}}.
\end{equation}
\end{lemma}
\begin{proof}[Proof of $(\ref{a3n})$]	
Setting $t=3$ in (\ref{atgen}), we get
\begin{displaymath}
\sum\limits_{n=0}^{\infty }{{{a}_{3}}(n){{q}^{n}}}=\frac{(q^3;q^3)_{\infty}^{3}}{(q;q)_{\infty}}.
\end{displaymath}
Note that
\begin{equation}\label{expan}
(q;q)_{\infty}=(q;q^3)_{\infty}(q^2;q^3)_{\infty}(q^3;q^3)_{\infty},
\end{equation}
we have
\begin{equation}\label{a3nstart}
\sum\limits_{n=0}^{\infty}{a_{3}(n)q^n}=\frac{(q^3;q^3)_{\infty}^{2}}{(q;q^3)_{\infty}(q^2;q^3)_{\infty}}.
\end{equation}
Taking $(a,b,z,q) \rightarrow (q,q^{4},q,q^3)$ in (\ref{psi}), we obtain
\[\sum\limits_{n=-\infty}^{\infty}{\frac{(q;q^3)_{n}}{(q^4;q^3)_{n}}\cdot q^{n}}=\frac{(q^2,q,q^3,q^3;q^3)_{\infty}}{(q,q^2,q^4,q^2;q^3)_{\infty}}.\]
Dividing both sides by $1-q$, after simplification, we get
\begin{equation}\label{a3nmid}
\sum\limits_{n=-\infty}^{\infty}{\frac{q^n}{1-q^{3n+1}}}=\frac{(q^3;q^3)_{\infty}^{2}}{(q;q^3)_{\infty}(q^2;q^3)_{\infty}}.
\end{equation}
Combining (\ref{a3nstart}) with (\ref{a3nmid}), we obtain
\[\sum\limits_{n=0}^{\infty}{a_{3}(n)q^n}=\sum\limits_{n=-\infty}^{\infty}{\frac{q^n}{1-q^{3n+1}}}.\]
Replacing $q$ by $q^3$ and multiplying both sides by $q$, we get
\begin{equation}\label{a3nfinal}
\begin{split}
\sum\limits_{n=0}^{\infty}{a_{3}(n)q^{3n+1}}&=\sum\limits_{m=0}^{\infty}{\frac{q^{3m+1}}{1-q^{3(3m+1)}}}+\sum\limits_{m=-\infty}^{-1}{\frac{q^{3m+1}}{1-q^{3(3m+1)}}}  \\
&=\sum\limits_{m=0}^{\infty}{\frac{q^{3m+1}}{1-q^{3(3m+1)}}}+\sum\limits_{m=0}^{\infty}{\frac{q^{-3m-2}}{1-q^{3(-3m-2)}}}\\
&=\sum\limits_{m=0}^{\infty}{\frac{q^{3m+1}}{1-q^{3(3m+1)}}}-\sum\limits_{m=0}^{\infty}{\frac{q^{2(3m+2)}}{1-q^{3(3m+2)}}}\\
&=\sum\limits_{m=0}^{\infty}\sum\limits_{k=0}^{\infty}{q^{(3m+1)(3k+1)}}-\sum\limits_{m=0}^{\infty}\sum\limits_{k=0}^{\infty}{q^{(3m+2)(3k+2)}},
\end{split}
\end{equation}
here the second equality follows by replacing $m$ by $-m-1$ in the second summation.

Now (\ref{a3n}) follows by comparing the coefficients of $q^{3n+1}$ on both sides of (\ref{a3nfinal}).
\end{proof}
Let $\sigma (n)$ denote the sum of positive divisors of $n$. Applying the method in proving (\ref{a3n}), we can find the explicit formula for $A_{3}(n)$.
\begin{theorem}\label{A3n}
For any integer $n\ge 0$, we have $A_{3}(n)=\frac{1}{3}\sigma(3n+2)$. If we write $3n+2=\prod\limits_{i=1}^{s}p_{i}^{\alpha_i}$ as the unique prime factorization, then
\[A_{3}(n)=\frac{1}{3}\prod\limits_{i=1}^{s}\frac{p_{i}^{\alpha_{i}+1}-1}{p_i-1}.\]
\end{theorem}
\begin{proof}
Setting $t=3$ in (\ref{At}), and applying (\ref{expan}) we obtain that
\begin{equation}\label{A3expan}
\sum\limits_{n=0}^{\infty}{A_{3}(n)q^n}=\frac{(q^3;q^3)_{\infty}^{4}}{(q;q^3)_{\infty}^{2}(q^2;q^3)_{\infty}^{2}}.
\end{equation}

Taking $(a,b,q)\rightarrow (q,q^4,q^3)$ in (\ref{Rama}), and dividing both sides by $1-\frac{q^2}{z}$, we obtain
\begin{equation}\label{mid1}
\sum\limits_{n=-\infty}^{\infty}{\frac{(q;q^3)_{n}}{(q^4;q^3)_{n}}\cdot \frac{z^n}{1-q^2/z}} =\frac{(qz,q^5/z,q^3,q^3;q^3)_{\infty}}{(z,q^3/z,q^4,q^2;q^3)_{\infty}}.
\end{equation}

Let $z\rightarrow q^2$. By L'Hospital's rule, we deduce that
\begin{displaymath}
\sum\limits_{n=-\infty}^{\infty}{\frac{(1-q)\cdot nq^{2n}}{1-q^{3n+1}}}=\frac{(q^3;q^3)_{\infty}^{4}}{(q^2;q^3)_{\infty}^{2}(q;q^3)_{\infty}(q^4;q^3)_{\infty}}.
\end{displaymath}
Dividing both sides by $1-q$ and combining with (\ref{A3expan}), we obtain
\begin{equation}\label{mid2}
\sum\limits_{n=0}^{\infty}{A_{3}(n)q^n}=\frac{(q^3;q^3)_{\infty}^{4}}{(q;q^3)_{\infty}^{2}(q^2;q^3)_{\infty}^{2}}=\sum\limits_{n=-\infty}^{\infty}{\frac{nq^{2n}}{1-q^{3n+1}}}.
\end{equation}
Replacing $q$ by $q^3$ and multiplying both sides by $q^2$, we see that
\begin{equation}\label{A3nfinal}
\begin{split}
\sum\limits_{n=0}^{\infty}{A_{3}(n)q^{3n+2}}&=\sum\limits_{m=0}^{\infty}{\frac{mq^{2(3m+1)}}{1-q^{3(3m+1)}}}+\sum\limits_{m=-\infty}^{-1}{\frac{mq^{2(3m+1)}}{1-q^{3(3m+1)}}}
\\ &=\sum\limits_{m=0}^{\infty}{\frac{mq^{2(3m+1)}}{1-q^{3(3m+1)}}}+\sum\limits_{m=0}^{\infty}{\frac{(-m-1)q^{2(-3m-2)}}{1-q^{3(-3m-2)}}}\\
&=\sum\limits_{m=0}^{\infty}{\frac{mq^{2(3m+1)}}{1-q^{3(3m+1)}}}+\sum\limits_{m=0}^{\infty}{\frac{(m+1)q^{3m+2}}{1-q^{3(3m+2)}}}\\
&=\sum\limits_{m=0}^{\infty}\sum\limits_{k=0}^{\infty}{mq^{(3m+1)(3k+2)}} +\sum\limits_{m=0}^{\infty}\sum\limits_{k=0}^{\infty}{(m+1)q^{(3m+2)(3k+1)}} \\
&=\frac{1}{3}\sum\limits_{m=0}^{\infty}\sum\limits_{k=0}^{\infty}{\Big((3m+1)q^{(3m+1)(3k+2)}+(3m+2)q^{(3m+2)(3k+1)}\Big)}\\
&\quad +\frac{1}{3}\sum\limits_{m=0}^{\infty}\sum\limits_{k=0}^{\infty}{\Big(q^{(3m+2)(3k+1)}-q^{(3m+1)(3k+2)}\Big)}\\
\end{split}
\end{equation}
Interchanging the roles of $k$ and $m$, we see that
\[\sum\limits_{m=0}^{\infty}\sum\limits_{k=0}^{\infty}{q^{(3m+2)(3k+1)}}=\sum\limits_{k=0}^{\infty}\sum\limits_{m=0}^{\infty}{q^{(3k+2)(3m+1)}}=\sum\limits_{m=0}^{\infty}\sum\limits_{k=0}^{\infty}{q^{(3m+1)(3k+2)}}.\]
Thus the second sum in the right hand side of (\ref{A3nfinal}) vanishes. Comparing the coefficients of $q^{3n+2}$ on both sides of (\ref{A3nfinal}), we prove the first assertion of the theorem. The second assertion then follows immediately.
\end{proof}

Once we know the explicit formula of $A_{3}(n)$, we can verify those identities in (\ref{BNid}) by simple arguments. For example, since $\sigma(n)$ is multiplicative, by Theorem \ref{A3n} we have
\[A_{3}(4n)=\frac{1}{3}\sigma(2(6n+1))=\frac{1}{3}\sigma(2)\sigma(6n+1)=\sigma(6n+1),\]
\[A_{3}\Big(2^{2k+2}n+\frac{2(2^{2k}-1)}{3}\Big)=\frac{1}{3}\sigma\big(2^{2k+1}(6n+1)\big)=\frac{1}{3}\sigma(2^{2k+1})\sigma(6n+1).\]
Note that $\sigma(2^{2k+1})=2^{2k+2}-1$, this proves the first identity in (\ref{BNid}). Others can be proved in a similar way.

Moreover, we can extend (\ref{BNid}) to some large families of arithmetic identities.
\begin{theorem}\label{Arelation}
Let $p$ be a prime and $k$, $n$ be nonnegative integers. \\
(1) If $p\equiv 1$ \text{\rm{(mod $3$)}}, we have
\[A_{3}\Big(p^kn+\frac{2p^{k}-2}{3}\Big)=\frac{p^{k}-1}{p-1}A_{3}\Big(pn+\frac{2p-2}{3}\Big)-\frac{p^{k}-p}{p-1}A_{3}(n).\]
(2) If $p\equiv 2$ \text{\rm{(mod $3$)}}, we have
\[A_{3}\Big(p^{2k}n+\frac{2p^{2k}-2}{3}\Big)=\frac{p^{2k}-1}{p^2-1}A_{3}\Big(p^2n+\frac{2p^2-2}{3}\Big)-\frac{p^{2k}-p^2}{p^2-1}A_{3}(n).\]
\end{theorem}
\begin{proof}
We write $3n+2=p^{m}N$, where $N$ is an integer not divisible by $p$.

(1) By Theorem \ref{A3n} we deduce that
\begin{equation}\label{Astart1}
A_{3}(n)=\frac{1}{3}\sigma(p^mN)=\frac{1}{3}\sigma(p^m)\sigma(N)=\frac{1}{3}\cdot \frac{p^{m+1}-1}{p-1}\sigma(N),
\end{equation}
Similarly we have
\begin{equation}\label{Astart2}
A_{3}\big(pn+\frac{2p-2}{3}\big)=\frac{1}{3}\sigma(p^{m+1}N)=\frac{1}{3}\cdot \frac{p^{m+2}-1}{p-1}\sigma(N),
\end{equation}
\begin{equation}\label{Astart3}
A_{3}\Big(p^kn+\frac{2p^k-2}{3}\Big)=\frac{1}{3}\sigma(p^{k+m}N)=\frac{1}{3}\cdot \frac{p^{k+m+1}-1}{p-1}\sigma(N).
\end{equation}
Now the assertion follows from (\ref{Astart1})--(\ref{Astart3}) by direct verification.

(2) In the same way we have
\begin{equation}\label{Astart4}
A_{3}\big(p^2n+\frac{2p^2-2}{3}\big)=\frac{1}{3}\sigma(p^{m+2}N)=\frac{1}{3}\cdot \frac{p^{m+3}-1}{p-1} \sigma(N),
\end{equation}
and
\begin{equation}\label{Astart5}
A_{3}\Big(p^{2k}n+\frac{2p^{2k}-2}{3}\Big)=\frac{1}{3}\sigma(p^{2k+m}N)=\frac{1}{3}\cdot \frac{p^{2k+m+1}-1}{p-1}\sigma(N).
\end{equation}
Combining (\ref{Astart1}), (\ref{Astart4}) and (\ref{Astart5}), we prove the assertion by direct verification.
\end{proof}

Setting $p=2,5,7$ in Theorem \ref{Arelation}, we obtain the following arithmetic identities for $k, n\ge 0$,
\begin{displaymath}
\begin{split}
A_{3}\Big(2^{2k}n+\frac{2^{2k+1}-2}{3}\Big)&=\frac{2^{2k}-1}{3}A_{3}(4n+2)-\frac{2^{2k}-4}{3}A_{3}(n),\\
A_{3}\Big(5^{2k}n+\frac{2\cdot 5^{2k}-2}{3}\Big)&=\frac{5^{2k}-1}{24}A_{3}(25n+16)-\frac{5^{2k}-25}{24}A_{3}(n),\\
A_{3}\Big(7^{k}n+\frac{2\cdot 7^k-2}{3}\Big)& =\frac{7^k-1}{6}A_{3}(7n+4)-\frac{7^k-7}{6}A_{3}(n).
\end{split}
\end{displaymath}

\begin{theorem}\label{Arelation}
Let $p$ be a prime and $k, n$ be nonnegative integers such that $p\nmid 3n+2$. \\
(1) If $p\equiv 1$ \text{\rm{(mod $3$)}}, we have
\[A_{3}\Big(p^kn+\frac{2p^{k}-2}{3}\Big)=\frac{p^{k+1}-1}{p-1}A_{3}(n).\]
(2) If $p\equiv 2$ \text{\rm{(mod $3$)}}, we have
\[A_{3}\Big(p^{2k}n+\frac{2p^{2k}-2}{3}\Big)=\frac{p^{2k+1}-1}{p-1}A_{3}(n).\]
\end{theorem}
\begin{proof}
From Theorem \ref{A3n}, we deduce that
\[A_{3}\Big(p^kn+\frac{2p^{k}-2}{3}\Big)=\frac{1}{3}\sigma\big(p^k(3n+2)\big)=\frac{1}{3}\sigma(p^k)\sigma(3n+2)=\frac{p^{k+1}-1}{p-1}A_{3}(n).\]
This implies (1). (2) can be proved in a similar way.
\end{proof}
For example, let $p=2$ and replacing $n$ by $2n+1$ in (2), we obtain the third identity of (\ref{BNid}).
If we set $p=5$ (resp.\ $p=7$) and replace $n$ by $5n+r$ (resp. \ $7n+r$), we deduce that for $k,n \ge 0$,
\[A_{3}\Big(5^{2k}(5n+r)+\frac{2\cdot 5^{2k}-2}{3}\Big)=\frac{5^{2k+1}-1}{4}A_{3}(5n+r), \quad  r\in \{0,2,3,4\}\]
and
\[A_{3}\Big(7^{k}(7n+r)+\frac{2\cdot 7^k-2}{3}\Big)=\frac{7^{k+1}-1}{6}A_{3}(7n+r), \quad r\in \{0,1,2,3,5,6\}.\]

We conclude this section by proving Conjecture \ref{Xiaconj}.
\begin{proof}[Proof of Conjecture \ref{Xiaconj}]
By Theorem \ref{A3n}, we get
\begin{equation}\label{Amid1}
\begin{split}
A_{3}\Big(4^{k_{0}(\alpha+1)}n+\frac{2^{2k_{0}(\alpha+1)-1}-2}{3}\Big)&=\frac{1}{3}\sigma\big(2^{2k_{0}(\alpha+1)-1}(6n+1)\big)\\
&=\frac{2^{2k_{0}(\alpha+1)}-1}{3}\sigma(6n+1).
\end{split}
\end{equation}
Let $k_{0}=\frac{1}{2}p^{j}(p-1)$. Since $2k_{0}(\alpha+1) \equiv 0$ (mod $p^{j}(p-1)$), by Euler's theorem, we have
$2^{2k_{0}(\alpha+1)} \equiv 1$ (mod $p^{j+1}$). From (\ref{Amid1}) the conjecture follows immediately.
\end{proof}
Indeed, most of the congruences found by Xia \cite{XiaRama} can be proved by using Theorem \ref{A3n}. We omit the details here.

\section{Explicit Formula for $B_{3}(n)$}
In order to find the explicit formula for $B_{3}(n)$, we need the following formula.
\begin{lemma}[Bailey's ${}_{6}\psi_{6}$ formula]
For $|qa^2/(bcde)|<1$,
\begin{equation}\label{Bailey}
\begin{split}
&\quad {}_{6}\psi_{6}\Big(\begin{matrix} q\sqrt{a},& -q \sqrt{a},& b,& c,& d, &e\\ \sqrt{a},& -\sqrt{a},& aq/b, &aq/c, &aq/d, &aq/e \end{matrix}; q, \frac{qa^2}{bcde}\Big)\\
&=\frac{(aq,aq/(bc),aq/(bd),aq/(be),aq/(cd),aq/(ce),aq/(de),q,q/a;q)_{\infty}}{(aq/b,aq/c,aq/d,aq/e,q/b,q/c,q/d,q/e,qa^2/(bcde);q)_{\infty}}.
\end{split}
\end{equation}
\end{lemma}
For the proof of this lemma, see \cite[Sec. 5.3]{Gasper}.
\begin{theorem}\label{B3n}
For any integer $n\ge 0$, we have
\[B_{3}(n)=\sum\limits_{\begin{smallmatrix} d|n+1 \\ d\equiv 1 \, \text{\rm{(mod 3)}} \end{smallmatrix}}{\Big(\frac{n+1}{d}\Big)^2}-\sum\limits_{\begin{smallmatrix}d|n+1 \\ d\equiv 2 \, \text{\rm{(mod 3)}} \end{smallmatrix}}{\Big(\frac{n+1}{d}\Big)^2}.\]
Furthermore, if we write
\[n+1=3^{\alpha}\prod\limits_{p_i \equiv 1 \, \mathrm{(mod \, 3)}}{p_{i}^{\alpha_i}}\prod\limits_{q_{j} \equiv 2 \, \mathrm{(mod \,3)}}{q_{j}^{\beta_j}}\]
as the unique prime factorization of $n+1$ with $\alpha, \alpha_{i}, \beta_{j} \ge 0$, then
\[B_{3}(n)=3^{2\alpha}\prod\limits_{p_i \equiv 1 \, \mathrm{(mod \,3)}}{\frac{p_{i}^{2(\alpha_i+1)}-1}{p_{i}^2-1}}\prod\limits_{q_{j} \equiv 2 \, \mathrm{(mod \, 3)}}{\frac{q_{j}^{2\beta_j+2}+(-1)^{\beta_j}}{q_{j}^2+1}}.\]
\end{theorem}
\begin{proof}
Setting $t=3$ in (\ref{Bt}) and applying (\ref{expan}), we see that
\begin{equation}\label{B3expan}
\sum\limits_{n=0}^{\infty}{B_{3}(n)q^n}=\frac{(q^3;q^3)_{\infty}^{6}}{(q;q^3)_{\infty}^{3}(q^2;q^3)_{\infty}^{3}}.
\end{equation}

Taking $(a,b,c,d,e,q)\rightarrow (q^2,q,q,q,q,q^3)$ in (\ref{Bailey}), then multiplying both sides by $\frac{q(1-q^2)}{(1-q)^4}$, we obtain
\begin{displaymath}
\sum\limits_{n=-\infty}^{\infty}{\frac{(1+q^{3n+1})q^{3n+1}}{(1-q^{3n+1})^3}}=q\cdot \frac{(q^3;q^3)_{\infty}^{6}}{(q;q^3)_{\infty}^{3}(q^2;q^3)_{\infty}^{3}}.
\end{displaymath}
Combining this with (\ref{B3expan}), we deduce that
\begin{equation}\label{B3key}
\begin{split}
\sum\limits_{n=0}^{\infty}{B_{3}(n)q^{n+1}}&=\sum\limits_{m=0}^{\infty}{\frac{q^{3m+1}(1+q^{3m+1})}{(1-q^{3m+1})^3}}+\sum\limits_{m=-\infty}^{-1}{\frac{q^{3m+1}(1+q^{3m+1})}{(1-q^{3m+1})^3}} \\
&=\sum\limits_{m=0}^{\infty}{\frac{q^{3m+1}(1+q^{3m+1})}{(1-q^{3m+1})^3}}-\sum\limits_{m=0}^{\infty}{\frac{q^{3m+2}(1+q^{3m+2})}{(1-q^{3m+2})^3}},
\end{split}
\end{equation}
here the second equality follows by replacing $m$ by $-m-1$ in the second sum.

It is well known that
\[\frac{1}{1-x}=\sum\limits_{k=0}^{\infty}{x^k}, \quad |x|<1.\]
Applying the operator $x\frac{d}{dx}$ twice to both sides, we get
\[\frac{x(1+x)}{(1-x)^3}=\sum\limits_{k=1}^{\infty}{k^2x^k}, \quad |x|<1.\]
Applying this identity to (\ref{B3key}), we obtain
\[\sum\limits_{n=0}^{\infty}{B_{3}(n)q^{n+1}}= \sum\limits_{m=0}^{\infty}\sum\limits_{k=1}^{\infty}{k^2\big(q^{(3m+1)k}-q^{(3m+2)k}\big)}.\]
The first assertion of this theorem now follows immediately by comparing the coefficients of $q^{n+1}$ on both sides.

Let
\[f(n)=\sum\limits_{\begin{smallmatrix} d|n \\ d\equiv 1 \, \text{\rm{(mod 3)}} \end{smallmatrix}}{\Big(\frac{n}{d}\Big)^2}-\sum\limits_{\begin{smallmatrix}d|n \\ d\equiv 2 \, \text{\rm{(mod 3)}} \end{smallmatrix}}{\Big(\frac{n}{d}\Big)^2}.\]
Suppose $m$ and $n$ are integers which are coprime to each other. It is not hard to see that
\begin{displaymath}
\begin{split}
 f(mn)&=\sum\limits_{\begin{smallmatrix} d|mn \\ d\equiv 1 \, \text{\rm{(mod 3)}} \end{smallmatrix}}{\Big(\frac{mn}{d}\Big)^2}-\sum\limits_{\begin{smallmatrix}d|mn \\ d\equiv 2 \, \text{\rm{(mod 3)}} \end{smallmatrix}}{\Big(\frac{mn}{d}\Big)^2} \\
&=\sum\limits_{\begin{smallmatrix} d_{1}|m \\ d_{1}\equiv 1 \, \text{\rm{(mod 3)}} \end{smallmatrix}}\sum\limits_{\begin{smallmatrix} d_{2}|n \\ d_{2}\equiv 1 \, \text{\rm{(mod 3)}} \end{smallmatrix}}+\sum\limits_{\begin{smallmatrix}d_{1}|m \\ d_{1}\equiv 2 \, \text{\rm{(mod 3)}} \end{smallmatrix}}\sum\limits_{\begin{smallmatrix}d_{2}|n \\ d_{2}\equiv 2 \, \text{\rm{(mod 3)}} \end{smallmatrix}}{\Big(\frac{mn}{d_1d_2}\Big)^2}\\
&\quad -\sum\limits_{\begin{smallmatrix} d_{1}|m \\ d_{1}\equiv 1 \, \text{\rm{(mod 3)}} \end{smallmatrix}}\sum\limits_{\begin{smallmatrix}d_{2}|n \\ d_{2}\equiv 2 \, \text{\rm{(mod 3)}} \end{smallmatrix}}-\sum\limits_{\begin{smallmatrix}d_{1}|m \\ d_{1}\equiv 2 \, \text{\rm{(mod 3)}} \end{smallmatrix}}
\sum\limits_{\begin{smallmatrix} d_{2}|n \\ d_{2}\equiv 1 \, \text{\rm{(mod 3)}} \end{smallmatrix}}{\Big(\frac{mn}{d_1d_2}\Big)^2}\\
&=\Big(\sum\limits_{\begin{smallmatrix} d_{1}|m \\ d_{1}\equiv 1 \, \text{\rm{(mod 3)}} \end{smallmatrix}}{\Big(\frac{m}{d_{1}}\Big)^2}-\sum\limits_{\begin{smallmatrix}d_{1}|m \\ d_{1}\equiv 2 \, \text{\rm{(mod 3)}} \end{smallmatrix}}{\Big(\frac{m}{d_{1}}\Big)^2}\Big) \\
&\quad \cdot \Big(\sum\limits_{\begin{smallmatrix} d_{2}|n \\ d_{2}\equiv 1 \, \text{\rm{(mod 3)}} \end{smallmatrix}}{\Big(\frac{n}{d_{2}}\Big)^2}-\sum\limits_{\begin{smallmatrix}d_{2}|n \\ d_{2}\equiv 2 \, \text{\rm{(mod 3)}} \end{smallmatrix}}{\Big(\frac{n}{d_{2}}\Big)^2}\Big) \\
&=f(m)f(n)
\end{split}
\end{displaymath}
This implies that $f(n)$ is multiplicative. For any prime $p$, from the definition of $f(n)$ and by direct calculations, we obtain that
\begin{displaymath}
f(p^k)=\left\{\begin{array}{ll}
3^{2k} & \textrm{if $p=3$} \\
\frac{p^{2(k+1)}-1}{p^2-1} & \textrm{if $p \equiv 1$ (mod 3)} \\
\frac{p^{2k+2}+(-1)^k}{p^2+1} & \textrm{if $p \equiv 2$ (mod 3)}.
\end{array}\right.
\end{displaymath}
The second assertion of this theorem then follows since $f(n)$ is multiplicative and $B_{3}(n)=f(n+1)$.
\end{proof}

\begin{theorem}\label{B3relation1}
Let $p$ be a prime and $k, n$ be nonnegative integers. \\
(1) If $p \equiv 1$ \text{\rm{(mod 3)}}, we have
\[B_{3}\Big(p^{k}n+p^{k}-1\Big)=\frac{p^{2k}-1}{p^2-1}B_{3}(pn+p-1)-\frac{p^{2k}-p^2}{p^2-1}B_{3}(n).\]
(2) If $p\equiv 2$ \text{\rm{(mod 3)}}, we have
\[B_{3}(p^{k}n+p^k-1)=\frac{p^{2k}-(-1)^{k}}{p^2+1}B_{3}(pn+p-1)+\frac{p^{2k}+(-1)^{k}p^2}{p^2+1}B_{3}(n).\]
\end{theorem}
\begin{proof}
Let $n+1=p^{m}N$, where $N$ is not divisible by $p$.

(1) Since $f(n)$ is multiplicative, we have
\begin{equation}\label{B3mid1}
\begin{split}
B_{3}(n)&=f(n+1)=f(p^m)f(N)=\frac{p^{2(m+1)}-1}{p^2-1}f(N),\\
B_{3}(pn+p-1)&=f(p(n+1))=f(p^{m+1})f(N)=\frac{p^{2(m+2)}-1}{p^2-1}f(N),\\
B_{3}(p^kn+p^k-1)&=f(p^k(n+1))=f(p^{k+m})f(N)=\frac{p^{2(m+k+1)}-1}{p^2-1}f(N).
\end{split}
\end{equation}
From those identities in (\ref{B3mid1}), we prove (1) by direct verification.

(2) Similarly, we have
\begin{equation}\label{B3mid2}
\begin{split}
B_{3}(n)&=f(n+1)=f(p^m)f(N)=\frac{p^{2(m+1)}+(-1)^m}{p^2+1}f(N),\\
B_{3}(pn+p-1)&=f(p(n+1))=f(p^{m+1})f(N)=\frac{p^{2(m+2)}+(-1)^{m+1}}{p^2+1}f(N),\\
B_{3}(p^kn+p^k-1)&=f(p^k(n+1))=f(p^{k+m})f(N)=\frac{p^{2(m+k+1)}+(-1)^{m+k}}{p^2+1}f(N).
\end{split}
\end{equation}
From those identities in (\ref{B3mid2}), we prove (2) by direct verification.
\end{proof}
By setting $p=2$ in this theorem we obtain (\ref{B3id3}) immediately. For more examples, by setting $p=5,7$ in this theorem, we obtain for $k,n \ge 0$,
\[B_{3}(5^kn+5^k-1)=\frac{5^{2k}-(-1)^k}{26}B_{3}(5n+4)+\frac{5^{2k}+25(-1)^k}{26}B_{3}(n)\]
and
\[B_{3}(7^kn+7^k-1)=\frac{7^{2k}-1}{48}B_{3}(7n+6)-\frac{7^{2k}-49}{48}B_{3}(n).\]

In some special cases, we can obtain some relations between $B_{3}(p^kn+p^k-1)$ and $B_{3}(n)$.
\begin{theorem}\label{B3relation2}
Let $p$ be a prime and $k, n$ be nonnegative integers. \\
(1) If $p=3$, we have $B_{3}(3^kn+3^k-1)=3^{2k}B_{3}(n)$. \\
(2) If $p\equiv 1$ \text{\rm{(mod 3)}} and $p \nmid n+1$, then
\[B_{3}(p^kn+p^k-1)=\frac{p^{2(k+1)}-1}{p^2-1}B_{3}(n).\]
(3) If $p\equiv 2$ \text{\rm{(mod 3)}} and $p \nmid n+1$, then
\[B_{3}(p^kn+p^k-1) =\frac{p^{2(k+1)}+(-1)^k}{p^2+1}B_{3}(n).\]
\end{theorem}
\begin{proof}
 Let $n+1=p^mN$, where $N$ is not divisible by $p$.  By Theorem \ref{B3n} and the fact that $f(n)$ is multiplicative, we get
 \[B_{3}(n)=f(p^mN)=f(p^m)f(N)\]
 and
 \[B_{3}(p^kn+p^k-1)=f(p^{k+m}N)=f(p^{k+m})f(N).\]

(1) We have
\[B_{3}(3^kn+3^k-1)=3^{2k+2m}f(N)=3^{2k}B_{3}(n).\]

(2) Since $p \nmid n+1$, we have $m=0$ and
\[B_{3}(p^kn+p^k-1)=f(p^{k})f(N)=\frac{p^{2(k+1)}-1}{p^2-1}B_{3}(n).\]

(3)  Since $p \nmid n+1$, we have $m=0$ and
\[B_{3}(p^kn+p^k-1) =f(p^k)f(N)=\frac{p^{2(k+1)}+(-1)^k}{p^2+1}B_{3}(n).\]
\end{proof}

Note that in this theorem, (1) is (\ref{B3id1}) exactly. By setting $p=2$ and replacing $n$ by $2n$ in (3), we obtain (\ref{B3id2}) at once.
For more examples, by setting $p=5$ (resp. \ $p=7$) and replacing $n$ by $5n+r$ (resp. $7n+r$) in (3) (resp. (2)) we obtain for $k,n \ge 0$,
\[B_{3}\Big(5^{k+1}n+5^k(r+1)-1\Big)=\frac{5^{2k+2}+(-1)^k}{26}B_{3}(5n+r), \quad r \in \{0,1,2,3\}\]
and
\[B_{3}\Big(7^{k+1}n+7^k(r+1)-1\Big)=\frac{7^{2k+2}-1}{48}B_{3}(7n+r), \quad r\in \{0,1,2,3,4,5\}\]
respectively.

\subsection*{Acknowledgements}
The author thanks Professor Chan Heng Huat for showing him the Lambert series representations of  the generating functions for $A_3(n)$ and $B_3(n)$.

\end{document}